\documentclass[10pt,twoside]{amsart}

\usepackage[english]{babel}
\usepackage[all]{xy}
\usepackage{amsfonts, amsthm, amssymb, amscd, amstext, amsmath, url,xr}
\usepackage{bbm}
\usepackage{stmaryrd}
\usepackage[mathscr]{euscript}
\usepackage{mathrsfs}
\usepackage{dsfont}
\usepackage{textcomp}
\usepackage{enumitem}
\usepackage[colorlinks=true,pagebackref=true]{hyperref} 

\title{}

\author{Mattia Cavicchi}
\address{Université Bourgogne Europe, CNRS, IMB UMR 5584, 21000 Dijon, France}
\email{mattia.cavicchi@ube.fr}
\keywords{Hecke algebras, motives of Shimura varieties, algebraic cycles}
\subjclass[2020]{Primary:14G35;14C15; Secondary:11F03;11G18;} 

\date{}

\newtheorem{thm}{Theorem}[section]
\newtheorem{prop}[thm]{Proposition}
\newtheorem{lm}[thm]{Lemma}
\newtheorem{cor}[thm]{Corollary}
\newtheorem{thmi}{Theorem}

\newtheorem{ques}{Question}

\theoremstyle{definition}
\newtheorem{dfi}[thm]{Definition}
\newtheorem{rem}[thm]{Remark}

\newtheorem{conv}{Convention}

\theoremstyle{definition}

\numberwithin{equation}{section}
\title[The motivic Hecke algebra for PEL Shimura varieties]{The motivic Hecke algebra for PEL Shimura varieties}

\makeatletter
\DeclareRobustCommand*\cal{\@fontswitch\relax\mathcal}
\makeatother

\date{\today}

\DeclareMathOperator{\DM}{DM} %derived motives
\newcommand{\an}{\mathrm{an}}

\DeclareMathOperator{\Hom}{Hom}
\DeclareMathOperator{\End}{End}

\newcommand{\Spec}[1]{\operatorname{Spec}(#1)}

\DeclareMathOperator{\Der}{D}

\newcommand{\QQ} {\mathbb Q}

\newcommand{\CC} {\mathbb C}
\renewcommand{\AA} {\mathbb A}
\newcommand{\PP} {\mathbb P}

\newcommand{\un}{\mathbbm 1} %motif constant

\newcommand{\et}{\mathrm{\acute{e}t}}

\newcommand{\BA}{{\mathbb{A}}}

\newcommand{\BC}{{\mathbb{C}}}

\newcommand{\BL}{{\mathbb{L}}}

\newcommand{\BP}{{\mathbb{P}}}
\newcommand{\BQ}{{\mathbb{Q}}}
\newcommand{\BR}{{\mathbb{R}}}

\newcommand{\BZ}{{\mathbb{Z}}}

\newcommand{\Fh}{{\mathfrak{h}}}

\newcommand{\Fp}{{\mathfrak{p}}}

\newcommand{\CB}{{\cal B}}

\newcommand{\CN}{{\cal N}}

\newcommand{\CV}{{\cal V}}

\newcommand{\one}{\mathds{1}}

\newcommand{\Sym}{\mathop{\rm Sym}\nolimits}
\newcommand{\Rep}{\mathop{\rm Rep}\nolimits}
\newcommand{\Res}{\mathop{\rm Res}\nolimits}

\newcommand{\id}{\mathop{\rm id} \nolimits}

\newcommand{\VHS}{\mathop{\rm VHS} \nolimits}
\newcommand{\Et}{\mathop{\rm Et} \nolimits}

\newcommand{\CH}{\mathop{\rm CH}\nolimits}

\newcommand{\Id}{\mathop{\rm Id}\nolimits}

\newcommand{\GL}{\mathop{\rm GL} \nolimits}

\newcommand{\GSp}{\mathop{\rm GSp} \nolimits}

\newcommand{\GU}{\mathop{\rm GU} \nolimits}

\newcommand{\Gr}{\mathop{\rm Gr} \nolimits}

\newcommand{\op}{\mathop{\rm \tiny{op}} \nolimits}
\newcommand{\ab}{\mathop{\rm \tiny{ab}} \nolimits}

\newcommand{\q}[1]{``#1''}
  
\newdir^{ (}{{}*!/-3pt/\dir^{(}}
\newdir_{ (}{{}*!/-3pt/\dir_{(}}
\newdir^{ )}{{}*!/+3pt/\dir^{)}}
\newdir_{ )}{{}*!/+3pt/\dir_{)}}

\def\CHM{\mathop{\rm CHM}\nolimits}

\def\Hom{\mathop{\rm Hom}\nolimits}

\def\dim{\mathop{\rm dim}\nolimits}

\begin{document}

	\maketitle
	\begin{abstract}
We construct a motivic lift of the action of the Hecke algebra on the cohomology of PEL Shimura varieties $S_K$. To do so, when $S_K$ is associated with a reductive algebraic group $G$ and $V$ is a local system on $S_K$ coming from a $G$-representation, we define a motivic Hecke algebra $\mathcal{H}^M(G,K)$ as a natural sub-algebra of the endomorphism algebra, in the triangulated category of motives, of the constructible motive associated with $S_K$ and $V$. The algebra $\mathcal{H}^M(G,K)$ is such that realizations induce an epimorphism from it onto the classical Hecke algebra. We then consider Wildeshaus's theory of interior motives, along with the necessary hypotheses for it to be employed. Whenever those assumptions hold, one gets a Chow motive realizing to interior $V$-valued cohomology of $S_K$, equipped with an action of $\mathcal{H}^M(G,K)$ as an algebra of correspondences modulo rational equivalence. We give a list of known cases where this applies. 
		\end{abstract}
		
			\tableofcontents

	\newpage
\section{Introduction}

Ever since Scholl's definition of motives for modular forms \cite{Sch90}, one of the frustrating features of the available constructions of motives for automorphic forms is the fact that they are only able to produce \emph{homological} motives, i.e., very roughly speaking, direct factors of smooth projective varieties, cut out by algebraic correspondences which are idempotent modulo \emph{homological equivalence}. For arithmetic applications, and especially for the study of special values of $L$-functions according to Beilinson's conjectures \cite{Nek94}, one would rather have \emph{Chow} motives at their own disposal, with \emph{rational} equivalence replacing the homological one\footnote{Notice for example that regulators are defined on spaces of higher cycles which are homologically trivial.}. The aim of this paper is to provide a hopefully useful tool to study these questions: namely, to construct, in some generality, an explicit algebra of correspondences modulo rational equivalence, where the hoped-for idempotents should live. 

Let us be more precise. When one replaces modular curves, seen as moduli spaces of elliptic curves with level structures, by moduli spaces of higher dimensional abelian varieties, equipped with a level structure as well as with a polarization of a given type and an action of some prescribed algebra of endomorphisms, one obtains higher dimensional Shimura varieties $S_K$ \emph{of PEL type}, defined over a number field $E$. They come associated with a reductive algebraic group $G$ over $\BQ$, whose finite-dimensional representations provide natural local systems $V$ on the complex variety $S_K(\BC)$. The singular cohomology spaces of the latter carry a natural action of algebras $\mathcal{H}(G,K)$ of Hecke operators of \emph{level} $K \subseteq G(\BA_f)$ - which we will always suppose to be \emph{small enough} for $S_K$ to be smooth. A similar action exists on the étale cohomology of the $\ell$-adic analogues of $V$. In order to study such actions, we have nowadays at our disposal not only the triangulated category $\DM_c(S)$ of \emph{constructible motives} with rational coefficients over general bases $S$, but also, because of the last 15 years of progresses in the motivic theory, objects lying in it which are relevant for our purposes. First, the six functor formalism à la Ayoub/Cisinski-Déglise for constructible motives (for which we use here \cite{CD19} as our reference), coupled with a theorem of Ancona \cite{Anc15}, provides objects $M(S_K,V)$ in $\DM_c(E)$, whose realizations coincide with the cohomology spaces $H^{\bullet}(S_K,V)$. Second, the relation of morphisms in $\DM_c(E)$ with $K$-theory tells us that the spaces $\End_{\DM_c(E)}(M(S_K,V))$ are the correct generalizations of Chow groups $\CH^*(X \times X)_{\BQ}$ of algebraic correspondences modulo rational equivalence on a smooth, projective variety $X$. 

With these notations, the main result of this paper is then the following. 

\begin{thmi}{(Theorem \ref{realhecke})} \label{thmA}
Let $S_K$ be a Shimura variety of PEL type, with underlying group $G$. For any representation $V$ of $G$, there exists a canonical sub-algebra 
\[
\mathcal{H}^M(G,K) \hookrightarrow \End_{\DM_c(E)}(M(S_K,V))
\]
such that realization induces an algebra epimorphism 
\[
\mathcal{H}^M(G,K) \twoheadrightarrow \mathcal{H}(G,K).
\]
\end{thmi}
In the above statement, one can fix ideas by taking as realization the \emph{Betti realization}. The algebra $\mathcal{H}^M(G,K)$ (the \emph{motivic Hecke algebra} of the title) is then a lift of the classical Hecke algebra, acting on singular cohomology $H^{\bullet}(S_K(\BC), V)$, to an algebra acting on motives. Some special cases of this theorem were known (in particular, in the case of modular curves and Hilbert modular varieties, by \cite{Wil12}), but the construction is already new when $S_K$ is a Siegel threefold, i.e. a moduli space of polarized abelian surfaces, corresponding to $G=\GSp_4$.

In order to explain the relation of this theorem with the problems mentioned at the beginning, let us now describe another instance of recent motivic advances. In a series of papers culminating with \cite{Wil19}, Wildeshaus has developed criteria allowing to construct motives for automorphic forms. Starting from \cite{Wil12}, those criteria have been shown to hold in a number of new cases beyond modular curves, thus making the first progress on these questions since Scholl's work  - the most recent one being due to the present author \cite{Cav19} and corresponding to PEL Shimura varieties for $G=\Res_{F \vert \BQ} \GSp_{4,F}$, for $F$ a totally real number field. Given a PEL Shimura variety $S_K$ and a system of coefficients $V$, Wildeshaus's key idea is to exploit Bondarko's theory of \emph{weight structures} \cite{Bon10} to functorially extract, under some conditions which will be made explicit in the main text (see Thm. \ref{intmotive_thm}), a \q{lowest weight-graded quotient} $\Gr_0 M(S_K,V)$ of the motive $M(S_K,V)$. By construction, the object $\Gr_0 M(S_K,V)$ is then an object of the full subcategory $\CHM(E) \hookrightarrow \DM_c(E)$ of \emph{Chow motives}; it is called \emph{interior motive} since its realizations coincide with \emph{interior cohomology}
\[
H^{\bullet}_!(S_K,V):=\mbox{Im}(H_c^{\bullet}(S_K,V) \rightarrow H^{\bullet}(S_K,V)).
\]
The Hecke algebra still acts on interior cohomology, and we can consider its image, denoted by $\mathcal{H}_!(G,K)$, in the endomorphisms of $H^{\bullet}_!(S_K,V)$. The functoriality properties of $\Gr_0 M(S_K,V)$ allow us to consider the image $\mathcal{H}_!^M(G,K)$ of $\mathcal{H}^M(G,K)$ in its endomorphism algebra, and to deduce the following directly from Theorem \ref{thmA}.
\begin{thmi}{(Theorem \ref{motheckeaction})} \label{thmB}
Let $S_K$ be a Shimura variety of PEL type, of underlying group $G$, and suppose that the representation $V$ of $G$ is such that the interior motive $\Gr_0 M(S_K,V)$ is defined. Then, there exists a canonical sub-algebra 
\[
\mathcal{H}_!^M(G,K) \hookrightarrow \End_{\CHM(E)}(\Gr_0 M(S_K,V))
\]
such that realization induces an algebra epimorphism 
\[
\mathcal{H}_!^M(G,K) \twoheadrightarrow \mathcal{H}_!(G,K).
\]
\end{thmi}
Let us stress that now, the algebra $\mathcal{H}_!^M(G,K)$ of the above theorem can be identified with a subalgebra of an algebra $\mathcal{C}$ of \emph{correspondences modulo rational equivalence}; i.e., some $\mathcal{C}$ arising as the image of $\CH^*(X \times X)_{\BQ}$ under \emph{some} idempotent element $p$, for \emph{some} smooth projective variety $X$ - the power of the method lying in the fact that the existence of $X$ and $p$ is granted without having to construct explicitly neither of them! Let us also observe that one very interesting case in which Theorem \ref{thmB} applies is the above-mentioned one of $G=\Res_{F \vert \BQ} \GSp_{4,F}$, when the coefficient system $V$ is \emph{regular}, because then, by the main result of \cite{Cav19}, the weight hypotheses under which the interior motive exists are satisfied.

We conclude by recalling that $H^{\bullet}_!(S_K,V)$ decomposes as a direct sum of simple $\mathcal{H}(G,K)$-modules, among which there are modules $\pi$ corresponding to \emph{cuspidal} cohomological automorphic representations of $G$. The associated idempotent elements $e(\pi)$ in $\mathcal{H}(G,K)$ cut out direct factors of the \emph{homological} motive underlying $\Gr_0 M(S_K,V)$, providing motives associated to those automorphic representations. Our results say nothing about the difficult question of whether the $e(\pi)$'s lift to \emph{idempotent} endomorphisms of $\Gr_0(M(S_K,V))$ - which would yield \emph{Chow} motives for automorphic forms - but they exhibit at least a natural, explicitly defined sub-algebra $\mathcal{H}^M(G,K)$ of endomorphisms containing \emph{some} lift of $e(\pi)$. We felt that this was a necessary step for all further investigations on these matters, and we expect the study of $\mathcal{H}^M(G,K)$, which we plan to pursue, to be interesting and helpful. 

\subsection{Plan of the paper}

Section \ref{chowbeilrev} serves as a reference for all the rest of the paper: it collects all we need to know and to use about constructible motives over a base and about the hearts of the weight structures on them, i.e. relative Chow motives. Special attention is paid to \emph{duality} on \emph{smooth} relative Chow motives, since some finer properties of the rigid tensor structure on those are needed later, for the proof of Proposition \ref{phi_g^e} and hence of our main result. Section \ref{motab} recalls the fundamental theorems on motives of abelian schemes due to Deninger-Murre, K\"{u}nnemann and Kings, which then intervene crucially in Section \ref{motcanconstr}: there, we gather basic facts and notations for PEL Shimura varieties $S_K$ with underlying group $G$, and then we pass to describe Ancona's construction of the motivic lifts of of sheaves on $S_K$ corresponding to representations of $G$, which makes use of certain algebras $\mathcal{B}$ of relative cycles, whose definition depends vitally on the results reviewed in Section \ref{motab}. In Section \ref{hecke}, we first recall the classical definition of Hecke algebras acting on the cohomology of Shimura varieties and give some of their properties; then, we arrive at the heart of the matter, i.e. the construction of the motivic Hecke algebra $\mathcal{H}^M(G,K)$. The key technical result is Theorem \ref{crucialcomm}, dealing with the universal abelian scheme $A_K$ over the PEL Shimura variety $S_K$ and with certain morphisms $\phi_g$, depending on $g \in G(\BA_f)$, between the motives of two different pullbacks of $A_K$ to the finite étale cover of $S_K$ determined by $g$. It states that any such morphism $\phi_g$ is equivariant with respect to the action of Ancona's algebra $\mathcal{B}$ on the two pullbacks. From this, we construct $\mathcal{H}^M(G,K)$ (Definition \ref{motheckealg}) and we deduce Theorem \ref{thmA} (Theorem \ref{realhecke} in the text). Finally, Section \ref{hecke_interior} first reviews Wildeshaus's theory of interior motives and explains how to deduce Theorem \ref{thmB} (Theorem \ref{motheckeaction} in the text) from Theorem \ref{thmA}, then provides the list of currently known cases in which Theorem \ref{thmB} applies (Theorem \ref{knowncases}).

\subsection{Relation with past work and future directions} The relation between our work and Wildeshaus's definition of Hecke endomorphisms on intersection motives of PEL Shimura varieties (\cite{Wil17}) - a definition which does not \emph{a priori} produce a motivic \emph{algebra} action -  is discussed in Remark \ref{morehecke}.\ref{differences} and in Remark \ref{intersection_motive}. To explain instead the possible continuations of this work, let us note that the classical Hecke algebra $\mathcal{H}(G,K)$ acting on cohomology is the image of a Hecke algebra $\mathcal{C}^{\infty}_c(G,K)$ of bi-invariant functions through a natural algebra homomorphism \eqref{hecke_arrow}. It is not clear to us whether this homomorphism factors through our motivic Hecke algebra $\mathcal{H}^M(G,K)$; we find this question stimulating, because a positive answer could help in making more explicit the composition law of $\mathcal{H}^M(G,K)$, hence shedding light on the structure of our algebra, even in cases where its construction was already known: most notably, in the case of classical modular curves. In particular, this could provide a further step towards solving the idempotent lifting problem, that was our starting motivation. 

As the referee suggested, it should also be possible to define, as in the classical case, a colimit motivic Hecke algebra acting on the motive of the Shimura variety at infinite level. We plan to study this in the future. 

\subsection{Notations and conventions} Throughout the paper, we will say \emph{base scheme} to mean a quasi-projective scheme $S$ over a characteristic zero field $E$, with algebraic closure $\bar{E}$. In the cases of interest for our main result, $E$ will actually be a number field. We will denote the complex analytic space associated to such a $S$ (under the choice of a complex embedding of $E$) by $S^{\an}$. 

If $X$ is a variety  (i.e., a reduced, separated finite type scheme over a field) which is moreover smooth and projective, then $\CH^{i}(X)$ stands for the Chow group \emph{with rational coefficients} of algebraic cycles of codimension $i$ on $X$. When $f:X \rightarrow Y$ is a morphism of smooth projective varieties, $\Gamma_f$ will denote both its graph and the class of the latter in $\CH^*(X \times Y)$. Under these hypotheses, the notation $^t Z$ for an algebraic cycle $Z$ in $X \times Y$ will denote (the class of) its \emph{transpose}, i.e. the pullback of $Z$ under the exchange of factors in $X \times Y$. 

When $S$ is a variety over $E$, the notation $H^{\bullet}(S, A)$ will be employed both for the singular cohomology spaces of $S^{\an}$ with values in some sheaf $A$ of vector spaces for the analytic topology and for the étale $\ell$-adic cohomology spaces of $S_{\bar{E}}$ with values in an étale $\ell$-adic sheaf $A$ (for some choice of prime $\ell$). Same conventions for cohomology with compact support $H^{\bullet}_c(S,A)$. We will use freely the six functor formalism available for the categories of sheaves just mentioned. If $f:X \rightarrow S$ is a morphism from a variety to a base scheme, then $H^i_S(X)$ will mean the $i$-th cohomology sheaf of $Rf_* \one_X$, with $\one_X$ being the constant sheaf, either in the analytic or in the étale $\ell$-adic case; we will even write $H^i(X)$ if the base is understood and the context makes it clear that we are considering the relative setting. The category of \emph{lisse} étale $\ell$-adic sheaves over $S$ will be denoted by $\Et_{\ell}(S)$, whereas $\VHS(S)$ will denote the category of polarizable (semisimple) variations of Hodge structure over $S^{\an}$. 

If $K$ is a subgroup of a group $G$, and $g \in G$, we will denote $K^g:=K \cap gKg^{-1}$. If $G$ is an algebraic group over $\BQ$, and $F$ a characteristic zero field, then $\Rep G_F$ will denote the category of algebraic representations of $G_F$ in finite-dimensional vector spaces over $F$. 

\section{Review of constructible motives and relative Chow motives} \label{chowbeilrev}

\subsection{Constructible motives over a base} For each base scheme $S$, we consider the triangulated categories $\DM_c(S,\QQ)$ of constructible motives over $S$, with rational coefficients. Various models for these categories exist: the reader can keep in mind the ones provided by \emph{constructible Beilinson motives} (\cite[Def. 15.1.1]{CD19}) or equivalently (\cite[Thm. 16.2.22]{CD19}) by the compact objects in the $\PP^1$-stable $\AA^1$-derived \'etale category. 

These categories are pseudo-abelian and symmetric monoidal, with tensor product denoted by $\otimes$ and unit given by the constant motive $\un_S$ over $S$. They satisfy the \emph{six functor formalism} as defined in \cite[A.5.1]{CD19}. In particular, one has \emph{cohomological motives} over $S$: given any morphism $f:X \rightarrow S$, we write:
$$
h_S(X):=f_*(\un_X).
$$
The motive $h_S(\BP^1_S)$ is canonically a direct sum of $\one_S$ and of a $\otimes$-invertible object $\BL_S$ called \emph{Lefschetz motive}, with tensor inverse denoted by $\one_S(1)$ and called \emph{Tate motive}\footnote{The notation we adopt for the Tate motive is the one usually employed when working with Chow motives. Indeed, most of the calculations involving $\one_S(1)$ will take place in the latter category. Nevertheless, we warn the reader that our notation differs by a shift from the most common one in the triangulated motivic setting - according to which, the object that we denote by $\one_S(1)$ would be rather denoted by $\one_S(1)[-2]$.}. We call $i$-th \emph{Tate twist} the operation of tensoring a motive $M$ by an integer power $\one_S(i):=\one_S(1)^{\otimes i}$, with resulting object denoted by $M(i)$. 

Whenever the base $S$ is clear from the context, we will drop the subscript $S$ and write $h(X)$, $\mathbb{L}$, $\one(i)$. 

\subsection{Chow motives over a base} By \cite[Thm. 3.3, Thm. 3.8 (i)-(ii)]{Héb11}, the category $\DM_c(S)$ is endowed with a canonical weight structure (in the sense of \cite[Def. 1.1.1]{Bon10}), called the \emph{motivic weight structure}, well-behaved with respect to the six functors, whose heart, called the category $\CHM(S)$ of \emph{Chow motives over} $S$, is generated (in the pseudo-abelian sense) by the motives $h_S(X)(p)$, with the morphism $X \rightarrow S$ proper with regular source, and $p \in \BZ$. By the main theorem of \cite{Fan16}, $\CHM(S)$ is canonically equivalent to the category of relative Chow motives introduced by Corti and Hanamura in \cite{CH00}, thus justifying its name. 

When $S$ is regular, we can in particular consider the full subcategory
\[
\CHM^s(S) \hookrightarrow \CHM(S)
\]
of \emph{smooth} Chow motives over $S$, generated by the $h_S(X)(p)$ as above but with $X \rightarrow S$ in addition \emph{smooth}. It coincides with the category defined in \cite[1.6]{DM91}. Spaces of morphisms between motives $h_S(X), h_S(Y)$ in $\CHM^s(S)$, with $X$ connected of relative dimension $d$ over $S$, verify then 
\begin{equation} \label{beilmorph}
\Hom_{\DM_c(S)}(h_S(X), h_S(Y)) \simeq \CH^{d}(X \times_{S} Y)
\end{equation}
(compatibly with composition). For $f : X \rightarrow Y$ a morphism of smooth projective schemes over $S$, the corresponding morphism $f^* : h_S(Y) \rightarrow h_S(X)$ in $\DM_c(S)$ is given, under \eqref{beilmorph}, by $^t \Gamma_f$. 

\subsection{Duality for smooth Chow motives} \label{duality} Let $S$ be regular. The category $\CHM^s(S)$ is a linear, pseudo-abelian, \emph{rigid} symmetric monoidal category, and we will denote by $M^{\vee}$ the dual of an object $M$ in $\CHM^s(S)$. By definition, there are \emph{evaluation} and \emph{coevaluation} morphisms
\[
\epsilon_M : M \otimes M^{\vee} \rightarrow \one_S, \ \ \ \eta_M : \one_S \rightarrow M^{\vee} \otimes M
\]
such that 
\[
(\epsilon_M \otimes \id_M) \circ (\id_M \otimes \eta_M) = \id_M, \ \ \ (\id_{M^{\vee}} \otimes \epsilon_M) \circ (\eta_M \otimes \id_{M^{\vee}}) = \id_{M^{\vee}}.
\]
Then, the adjunction isomorphism 
\begin{equation} \label{adj}
adj: \Hom_{\CHM^s(S)}(A \otimes B, C) \simeq \Hom_{\CHM^s(S)}(B, A^{\vee} \otimes C)
\end{equation}
is defined by sending $f : A \otimes B \rightarrow C$ to 
\[
B \xrightarrow[\eta_A \otimes \id_B] \ A^{\vee} \otimes A \otimes B \xrightarrow[\id_{A^{\vee}} \otimes f] \ A^{\vee} \otimes C
\]
with inverse $adj^{-1}$ sending $g : B \rightarrow A^{\vee} \otimes C$ to 
\[
A \otimes B \xrightarrow[\id_A \otimes g] \ A \otimes A^{\vee} \otimes C \xrightarrow[\epsilon_A \otimes \id_C] \ C.
\]
Moreover, given $f : A \rightarrow B$ in $\CHM^s(S)$, the dual map $f^{\vee}:B^{\vee} \rightarrow A^{\vee}$ is defined as the composition
\[
B^{\vee} \xrightarrow[\eta_A \otimes \id_{B^{\vee}}] \ A^{\vee} \otimes A \otimes B^{\vee} \xrightarrow[\id_{A^{\vee}} \otimes f \otimes \id_{B^{\vee}}] \ A^{\vee} \otimes B \otimes B^{\vee} \xrightarrow[id_{A^{\vee}} \otimes \epsilon_B] \ A^{\vee}.
\]
An useful reference for the above material is \cite[6.1]{AK02}). 

Then, one checks that the previous definitions imply the following formulae.
\begin{lm} \label{adjformulae}
Let $A, B, C, D$ be objects in $\mathcal{C}:=\CHM^s(S)$. 
\begin{enumerate}[wide, labelwidth=!, labelindent=0pt, label=(\arabic*)]
\item For any $\lambda \in \Hom_{\mathcal{C}}(A \otimes B, C)$ and $\mu \in \Hom_{\mathcal{C}}(C,D)$, we have
\[
adj(\mu \circ \lambda)=(\id_{A^{\vee}} \otimes \mu) \circ adj(\lambda).
\]
\label{f1}
\item For any $\lambda \in \Hom_{\mathcal{C}}(A, B)$ and $\mu \in \Hom_{\mathcal{C}}(B \otimes C,D)$, we have
\[
adj(\mu \circ (\lambda \otimes \id_C))=(\lambda^{\vee} \otimes \id_D) \circ adj(\mu).
\]
\label{f2}
\item For any $\lambda \in \Hom_{\mathcal{C}}(A, B)$ and $\mu \in \Hom_{\mathcal{C}}(B,C \otimes D)$, we have
\[
adj^{-1}(\mu \circ \lambda) = adj^{-1}(\mu) \circ (\id_{C^{\vee}} \otimes \lambda).
\]
\label{f3}
\item For any $\lambda \in \Hom_{\mathcal{C}}(A, B \otimes D)$ and $\mu \in \Hom_{\mathcal{C}}(B,C)$, we have
\[
adj^{-1} ((\mu \otimes \id_D) \circ \lambda) = adj^{-1} (\lambda) \circ (\mu^{\vee} \otimes \id_A).
\]
\label{f4}
\end{enumerate}
\end{lm}

\subsection{Coefficients}
The construction of $\DM_c(S)$ can be done taking an arbitrary $\BQ$-algebra $F$ as ring of coefficients instead of $\BQ$ (\cite[14.2.20]{CD19}), yielding triangulated, $F$-linear categories $\DM_c(S)_F$, such that the canonical functor $\DM_c(S) \otimes_{\BQ} F \rightarrow \DM_c(S)_F$ is fully faithful, and satisfying the $F$-linear analogues of the properties of $\DM_c(S)$. In particular, such categories are again pseudo-Abelian (\cite[Sect. 2.10]{Héb11}). The same then holds for the full subcategories $\CHM(S)_F$ and $\CHM^s(S)_F$, defined in the same way as in the case $F=\BQ$. All the properties stated in Subsection \ref{duality} hold, \emph{mutatis mutandis}, in $\CHM^s(S)_F$. 

\subsection{Realizations}

Given our assumptions on base schemes, we have at our disposal the following triangulated realization functors, commuting with the six functors: 
\begin{itemize}
\item for each complex embedding $\sigma$ of the base field $E$, the Betti realization:
$$
\rho_B:\DM_c(S,\QQ) \rightarrow \Der^b_c(S^\an,\QQ)
$$
whose target is the constructible derived category
 of rational sheaves over the analytic site of $S^\an=S^\sigma(\CC)$ (see \cite[end of Sec. 1]{CDN22} for a quick review of the definition of this functor);
 \item for each prime $\ell$, the $\ell$-adic realization:
$$
\rho_\ell:\DM_c(S,\QQ) \rightarrow \Der^b_c(S_\et,\QQ_\ell)
$$
whose target is the constructible derived category
 of Ekedahl's \'etale $\BQ_{\ell}$-sheaves
 (see \cite[7.2.24]{CD16}).
\end{itemize}
Actually, we will abuse of notation and denote by $\rho_B$ and $\rho_{\ell}$ the \emph{cohomological realization} functors, i.e. the composition of one of the above functors with the functor \q{direct sum of cohomology objects}
\begin{align*}
\Der^b_c(\mathcal{C}) & \rightarrow \mathcal{C} \\
K & \mapsto \oplus_i H^i(K).
\end{align*}
For a regular base $S$, and $\rho$ equal to any cohomological realization, the restriction of $\rho$ to the full subcategory $\CHM^s(S)$ admits a factorization, denoted in the same way, through the fully faithful embedding of a certain abelian subcategory $\mathcal{A}$ : 
\begin{itemize}
\item for $\rho=\rho_B$, one takes $\mathcal{A}$ equal to the category of $\BQ$-local systems over the analytic site of $S^{\an}$ ;
\item for $\rho=\rho_{\ell}$, one takes $\mathcal{A}$ equal to the category $\Et_{\ell}(S)$.
\end{itemize}
Moreover, $\rho_B$ has a factorization $\rho_H$, called the \emph{Hodge realization}, through the forgetful functor from $\VHS(S)$ towards $\BQ$-local systems. 

When speaking of \emph{realizations}, we will mean any one of the above functors (the cohomological ones with triangulated source, or, if the context permits, the above-described factorizations of their restrictions to smooth Chow motives). 

For any $\BQ$-algebra $F$ of coefficients, there are $F$-linear analogues of all these functors, between the $F$-linear versions of the relevant categories (i.e., the categories obtained by first tensoring morphisms with $F$ over $\BQ$, then taking the pseudo-abelian envelope). Such $F$-linear versions will be denoted by adding a subscript $F$ (e.g. $\VHS_F(S)$, $\Et_{F,\ell}(S)$)

\section{Motives of abelian schemes} \label{motab}

In this section, $A \rightarrow S$ will denote an abelian scheme over a base scheme $S$, of relative dimension $d$. We will write $\End(A)$ for its algebra of endomorphisms (as an abelian scheme), tensored with $\BQ$, and $n$ for the endomorphism of multiplication by an integer $n$. 

In the complex-analytic setting, the $i$-th cohomology sheaf $H^i(A)$ is endowed with the structure of a polarizable variation of Hodge structure over $S^{\an}$, of pure weight $i$. Then, recall the following. 

\begin{prop}{(\cite[4.4.3]{Del71b})} \label{abelhodge}
Fix an embedding $E \hookrightarrow \BC$. The functor $H^1$ induces an (anti-)equivalence of categories 
\[
\left\{ \begin{array}{cc}
\mbox{abelian schemes over} \ S_{\BC} \\
\mbox{modulo isogeny} 
\end{array} 
\right\}
\simeq \left\{ \begin{array}{cc}
\mbox{polarizable variations of} \ \BQ \mbox{-Hodge structure} \\
\mbox{over} \ S^{\an} \ \mbox{of type} \ (1,0), (0,1) 
\end{array} 
\right\}.
\]
\end{prop}

\begin{rem} \label{dualab_hodge}
Let us denote by $A^{\vee}$ the dual abelian scheme of $A$. Then, the polarizable variation of Hodge structure $H^1(A^{\vee})$ is canonically identified with $H^1(A)^{\vee}(-1)$. 
\end{rem}

We will need the following facts about motives of abelian schemes. 
\begin{thm}{(\cite[Thm. 3.1, Cor. 3.2]{DM91})}
\begin{enumerate}[wide, labelwidth=!, labelindent=0pt, label=(\arabic*)]
\item For each $i \in \{0, \cdots, 2d \}$ there exist canonical idempotents $\Fp_{A,i} \in \CH^d(A \times_{S} A)$ (called the \emph{Chow-Künneth projectors}) uniquely characterized by the equation
\begin{equation} \label{canchowkun}
^t \Gamma_n \circ \Fp_{A,i} = n^i \Fp_{A,i} = \Fp_{A,i} \circ \ ^t \Gamma_n
\end{equation}
in $\CH^d(A \times_{S} A)$. We call $i$-th \emph{Chow-Künneth component} of $h(A)$ (or of $A$) the direct factor of $h(A)$ in $\CHM^s(S)$ determined by $\Fp_{A,i}$, and we denote it by $h^i(A)$.  

\item In $\CHM^s(S)$, we have an isomorphism

\begin{equation} \label{CKdec}
h(A) \simeq \bigoplus\limits_{i=0}^{2d} h^i(A).
\end{equation}

\item If $\rho$ is either the Betti or the $\ell$-adic realization, we have, for each $i \in \{0, \dots, 2d \}$
\begin{equation} \label{realH1}
\rho(h^i(A))=H^i(A).
\end{equation}
\end{enumerate}
\end{thm}

\begin{rem}
\begin{enumerate}[wide, labelwidth=!, labelindent=0pt, label=(\arabic*)]
\item \label{transpchowkun} The Chow-Künneth projectors $\Fp_{A,i}$ are such that, for any $i$, the equation 
\[
^t \Fp_{A,i} = \Fp_{A,2d-i}
\]
holds (\cite[Remark 3, page 218]{DM91}). See \cite[proof of Thm. 3.1.1]{Kun94} for a proof. This implies that for all $0 \leq i \leq 2d$, there exists a canonical Poincaré duality isomorphism
\begin{equation} \label{pdiso}
h_S^{2d-i}(A)^{\vee} = h_S^i(A)(d).
\end{equation}
\item \label{functchowkun} For any homomorphism $f:A \rightarrow B$ of abelian schemes over $S$, we have 
\begin{equation*}
^t \Gamma_f \circ \Fp_{B,i} = \Fp_{A,i} \circ \ ^t \Gamma_f,
\end{equation*}
so that any such $f$ induces a map 
\begin{equation*}
f^*: h^i(B) \rightarrow h^i(A)
\end{equation*}
for all $i$ (\cite[Prop. 3.3]{DM91}). 
\item \label{isochowkun}
The previous two points imply that for any isogeny $f:A \rightarrow B$ of abelian schemes over $S$, we have 
\begin{equation*}
\Gamma_f \circ \Fp_{A,i} = \Fp_{B,i} \circ \ \Gamma_f,
\end{equation*}
so that any such $f$ induces a map 
\begin{equation*}
f_*: h^i(A) \rightarrow h^i(B).
\end{equation*}
\end{enumerate}
\label{propchowkun}
\end{rem}

\begin{rem} \label{dualab_mot}
The identification $H^1(A^{\vee}) = H^1(A)^{\vee}(-1)$ of Remark \ref{dualab_hodge} is the Hodge realization of a canonical isomorphism $h^1(A^{\vee})=h^1(A)^{\vee}(-1)$ in $\CHM^s(S)$. 
\end{rem}

For the following definition, note that for any motive $M$, for all $n$, the symmetric group $\mathfrak{S}_n$ acts on $M^{\otimes n}$. We denote by $\Sym^n M$ the image of the projector 
\begin{equation} \label{projsym}
\pi:= \frac{1}{n!} \sum\limits_{\sigma \in \mathfrak{S}_n} \sigma
\end{equation}
on $M^{\otimes n}$. 

\begin{thm}{\cite[Thm. 3.3.1]{Kun94}, \cite[p. 85]{Kun93}}
Let $A \rightarrow S$ be an abelian scheme of relative dimension $d$.
\begin{enumerate}[wide, labelwidth=!, labelindent=0pt, label=(\arabic*)]
\item For all $0 \leq i \leq 2d$, there exists a canonical isomorphism 
\begin{equation} \label{symiso}
\Sym^i h^1(A) \simeq h^i(A).
\end{equation}
\item
Any polarization of $A$ induces a Lefschetz isomorphism
\[
h^i(A) \simeq h^{2d-i}(A)(d-i).
\]
\end{enumerate}
\end{thm}

\begin{cor}
Combining point (2) above with the isomorphism \eqref{pdiso}, any polarization of $A$ induces an isomorphism 
\begin{equation} \label{I}
I : h^1(A) \simeq h^1(A)^{\vee}(-1).
\end{equation}
By adjunction (\eqref{adj}), the isomorphism $I$ induces a map
\begin{equation} \label{p}
p: h^1(A) \otimes  h^1(A) \rightarrow \BL
\end{equation}
and the isomorphism $I^{-1}$ induces a map
\begin{equation} \label{i}
\iota: \BL \rightarrow h^1(A) \otimes h^1(A).
\end{equation}
\end{cor}

\begin{rem}
 \label{sym}
The realizations of the isomorphisms \eqref{symiso} coincide indeed, for each $i$, with the classical isomorphisms
\[
\Lambda^i H^1 (A) \simeq H^i(A).
\]
The fact that symmetric powers of $h^1(A)$ realize to alternating powers of $H^1(A)$ is a consequence of the definition of the symmetry constraint of the tensor structure on motives, see \cite[Rem. 2.5]{Anc15} for a discussion. 
\end{rem}

\begin{prop}{\cite[Prop. 2.2.1]{Kin98}} \label{thm_kings}
The functor $h^1$ (on the category of abelian schemes over $S$) induces an isomorphism of $\BQ$-algebras 
\begin{equation} \label{isoEndH1}
\End(A)^{\op} \simeq \End_{\CHM^s(S)}(h^1(A)).
\end{equation}
\end{prop}

\section{PEL Shimura varieties and the motivic canonical construction} \label{motcanconstr}

In this section, we keep the notations concerning abelian schemes of the previous section. We will denote by $F$ an arbitrary number field. 

\subsection{Rational PEL data and PEL type Shimura varieties}  
A \emph{rational PEL datum} is a tuple $(V, B, *, \langle \cdot, \cdot \rangle, h_0)$ where 

- $B$ is a finite dimensional semisimple $\BQ$-algebra with positive involution $*$,

- $V$ is a finite dimensional $\BQ$-vector space on which $B$ acts,

- $\langle \cdot, \cdot \rangle$ is an alternating pairing on $V$,

- $h_0$ is a $\BR$-algebra homomorphism $\BC \rightarrow \End_{B_{\BR}}(V_{\BR})$.

These data are required to satisfy a series of conditions, which we will not recall here ; we refer the reader to \cite[5.1]{Lan17} for complete definitions and for more details on the facts that we are going to state.

\begin{dfi}
Let $(V, B, *, \langle \cdot, \cdot \rangle, h_0)$ be a rational PEL datum.
\begin{enumerate}[wide, labelwidth=!, labelindent=0pt, label=(\arabic*)]
\item We define the group $G$ underlying (or associated with) $(V, B, *, \langle \cdot, \cdot \rangle, h_0)$ as the (reductive) $\BQ$-algebraic group of automorphisms of $V$ commuting with $B$ and preserving the pairing $\langle \cdot, \cdot \rangle$ up to a \emph{rational} scalar.
\item If $G$ is associated with $(V, B, *, \langle \cdot, \cdot \rangle, h_0)$, we denote by $\BQ(1)$ the 1-dimensional representation of $G$ (called \emph{Tate twist}) on which $G$ acts by the character $g \mapsto gg^*$. For any representation $W$ of $G$, for any positive integer $n$, we denote $W(n):=W \otimes \BQ(1)^{\otimes n}$, $W(-n):=W \otimes (\BQ(1)^{\vee})^{\otimes n}$.
\end{enumerate}
\end{dfi}

\begin{rem} \label{tatetwistrep}
If $G$ is the group associated to a rational PEL datum $(V, B, *, \langle \cdot, \cdot \rangle, h_0)$, the pairing $\langle \cdot, \cdot \rangle$ induces a morphism of $G$-representations
\begin{equation} \label{canpairing}
V \otimes V \rightarrow \BQ(1) 
\end{equation}
and an isomorphism
\begin{equation} \label{candual}
V^{\vee} \simeq V(-1).
\end{equation}
\end{rem}

Letting $G(\BR)$ act naturally on $\Hom(\BC,\End_{B_{\BR}}(V_{\BR}))$, we put $X:=G(\BR) \cdot h_0$, and we will suppose from now on that $(G,X)$ is a \emph{Shimura datum}. It is then called a \emph{PEL-type} Shimura datum. In this case, for any compact open subgroup $K \subset G(\BA_f)$, the coset space
\begin{equation*}
G(\BQ) \backslash X \times G(\BA_f)/K
\end{equation*} 
coincides with the $\BC$-points of a quasi-projective algebraic variety $S_K$, called Shimura variety (of PEL type) associated with $G$ and of level $K$, canonically defined over a number field $E$ (called its \emph{reflex field}). 

\begin{conv}
For the rest of the paper, whenever we talk of a Shimura variety $S_K$, we will implicitly suppose that $K$ is small enough, so that $S_K$ is \emph{smooth}. 
\end{conv}

The varieties $S_K$ can be identified with (a disjoint union of connected components of) moduli spaces of abelian varieties $A$ equipped, in particular, with a polarization of a specific type induced by \eqref{candual}, with an injection $B \hookrightarrow \End(A)$, and with a suitable \emph{level structure} (the terminology PEL comes in fact from \emph{polarization, endomorphisms, level}). As a consequence, there exist \emph{universal abelian schemes} $A_K \rightarrow S_K$ of relative dimension $d=\frac{1}{2} \dim_{\BQ} V$, equipped with a polarization and with an injection of $\BQ$-algebras $B \hookrightarrow \End(A_K)$. 

\subsection{The canonical construction of motivic sheaves on PEL Shimura varieties}

Following the notations and conventions of the previous subsection, let $(G,X)$ be a PEL Shimura datum, with underlying rational PEL datum 
\[
(V, B, *, \langle \cdot, \cdot \rangle, h_0)
\]
and associated Shimura varieties $S_K$.

The connected components of $X$ are all homeomorphic to each other and contractible. They coincide with the universal cover of each connected component $S_K^{\an,+}$ of $S_K^{\an}$. Under our assumptions, any such $S_K^{\an,+}$ has fundamental group isomorphic to $G(\BQ) \cap gKg^{-1}$ for a suitable $g \in G(\BA_f)$. This provides, for any representation $V$ of $G_F$, a local system $\mu^K(V)$ of $F$-vector spaces on $S_K^{\an}$, functorially in $V$. The total space of the vector bundle associated with $\mu^K(V)$ is given by 
\begin{equation} \label{analyticloc}
G(\BQ) \backslash V(\BR) \times X \times G(\BA_f) / K
\end{equation}
with the obvious left action of $G(\BQ)$ on $V(\BR)$ defined by the representation $V$. 

By the very definition of a Shimura datum, this construction gives rise to the following.
\begin{dfi} \label{Hodgecan}
The \emph{Hodge canonical construction functor} is the exact tensor functor 
\begin{equation}
\mu_{H}^{K}: \Rep(G_F) \rightarrow \VHS_F(S_K).
\end{equation}
naturally enriching $\mu^K$ (cfr. \cite[1.18]{Pin90}).
\end{dfi}
An analogous construction can be carried out in the étale setting:
\begin{dfi} \label{lcan}
The \emph{$\ell$-adic canonical construction functor} is the natural functor (cfr. \cite[4.1]{Pin92})
\begin{equation*}
\mu_{\ell}^{K}: \Rep(G_F) \rightarrow \Et_{\ell, F}(S_K).
\end{equation*}
\end{dfi}

\begin{rem} \label{can&coh}
\begin{enumerate}[wide, labelwidth=!, labelindent=0pt, label=(\arabic*)]
\item By the construction of the universal abelian scheme $A_K$ over the PEL type Shimura variety $S_K$, each of $\mu^K_H$ and $\mu^K_{\ell}$ sends (in the appropriate category)
\[
V^{\vee} \mapsto H^1(A_K).
\]
\item Applying $\mu^K_H$ to the pairing \eqref{canpairing}, we get a polarization (in the sense of variations of Hodge structure on $S^{\an}_K$)
\begin{equation} \label{pairinghodge}
\langle \cdot, \cdot \rangle : H^1(A_K)^{\vee} \times H^1(A_K)^{\vee} \rightarrow \BQ(1).
\end{equation}
\end{enumerate}
\end{rem}

We now pass to define the algebra of relative cycles that will play a vital role in what follows. 
\begin{dfi} \label{Lef_alg}{\cite[Def. 5.2]{Anc15}}
Let $d$ be the relative dimension of the universal abelian scheme $A_K$ over $S_K$. For each positive integer $i$, the $F$-algebra $\CB_{i,F}$ is defined\footnote{Notice that we are adapting the original definition, given for general abelian schemes, to the more restricted context of \emph{loc. cit.}, Section 8.3.} as the sub-$F$-algebra
\begin{equation*}
\CB_{i,F} \hookrightarrow \End_{\CHM(S_K)_F}((h^1(A_K))^{\otimes i})
\end{equation*}
generated by the following (where we denote $B_F:=B \otimes_{\BQ} F$):
\begin{itemize}
\item the permutation group $\mathfrak{S}_i$,
\item the ring $B_F^{\op} \otimes \Id^{\otimes i-1}_{h^1(A_K)}$ (seen as a subalgebra of $\End_{\CHM(S_K)_F}((h^1(A_K))^{\otimes i}$ through Proposition \ref{thm_kings}),
\item if $i \geq 2$, the morphism $P \otimes \Id_{h^1(A_K)}^{\otimes (i-2)}$, where $P$ is defined, using the morphisms $\iota$ and $p$ associated in \eqref{i} and \eqref{p} with the given polarization on $A_K$, as the projector 
\begin{equation*}
\frac{1}{2d} \iota \circ p \in \End_{\CHM(S_K)_F}(h^1(A_K) \otimes h^1(A_K)).
\end{equation*}
\end{itemize}
\end{dfi}

\begin{dfi}{\cite[8.1]{Anc15}} \label{Lef_alg^r}
Let $A_K$ be as in the previous definition, and let $i,r$ be positive integers. Use the identifications 
\[
h^i(A_K^r) = \Sym^i h^1(A_K^r)
\]
\[
h^1(A_K^r)=h^1(A_K)^{\oplus r}
\]
and
\[
(h^1(A_K)^{\oplus r})^{\otimes i}=(h^1(A_K)^{\otimes i})^{\oplus r^i}
\]
to see the algebra $\CB_{i,F}^{\oplus r^i}$ as a subalgebra of $\End_{\CHM^s(S_K)}((h^1(A_K^r))^{\otimes i})$. Using the projector $\pi \in \End_{\CHM^s(S_K)}((h^1(A_K^r))^{\otimes i})$ of \eqref{projsym} (for the object $M=h^1(A_K^r)$), we define the sub-algebra $\CB_{i,r,F}$ of $\CB_{i,F}^{\oplus r^i}$ as
\[
\CB_{i,r,F} := \pi \circ \CB_{i,F}^{\oplus r^i} \circ \pi.
\]
\end{dfi} 

The algebra $\CB_{i,F}$ is the key for the following fundamental theorem, allowing one to lift the canonical construction functors to the motivic setting.
\begin{thm}{(\cite[Thms. 8.5-8.6]{Anc15}} \label{constmot}
There exists a $F$-linear monoidal\footnote{It is not a \emph{tensor} functor since it is not symmetric, because of the definition of the symmetry constraint on motives, cfr. Remark \ref{sym}.} functor 
\begin{equation}
\tilde{\mu}: \Rep(G_F) \rightarrow \CHM^s(S_K)_F
\end{equation}
called the \emph{motivic canonical construction}, commuting with Tate twists and sending $V$ to $h^1(A)(1)$, such that, if $\rho$ is the Hodge, resp. étale realization, then $\rho \circ \tilde{\mu}$ is isomorphic to $\mu^K_H$, resp. $\mu^K_{\ell}$. 

It induces isomorphisms of $F$-algebras
\begin{equation*}
\End_{\Rep (G_F)}((V^{\vee})^{\otimes i}) \simeq \CB_{i,F}
\end{equation*}
and
\begin{equation*} 
\End_{\Rep (G_F)}((\Lambda^i(V^{\vee}))^{\oplus r}) \simeq \CB_{i,r,F}.
\end{equation*}
\end{thm}

\begin{rem}
In \cite{Anc15}, the statement about realizations in the above theorem is only given for the Hodge realization. Nonetheless, one can see from the proof that the statement holds for the étale realization as well. Indeed, it suffices to observe that the $\ell$-adic analogue of \cite[Prop. 8.5]{Anc15} is verified.
\end{rem}

\begin{rem} \label{factpowab}
For every positive integer $r$, let $A_K^{r} \rightarrow S_K$ be the $r$-fold fibred product of $A_K$ with itself over $S_K$. Observe that since the group $G$ underlies a PEL Shimura datum, it is isomorphic over $\BR$ to a product of classical groups (\cite[pag. 51]{Lan17}); hence, the direct sum $V \oplus V^{\vee}$ of the standard representation $V$ with its dual generates the Tannakian category $\mbox{Rep}(G)$, by taking tensor products and direct summands. As a consequence, Theorem \ref{constmot} implies that when $F=\BQ$, every object in the essential image of $\tilde{\mu}$ is isomorphic (using the isomorphism \eqref{I} when necessary) to a finite direct sum $\bigoplus \limits_{j} M_j$, where each $M_j$ is a direct factor of a Tate twist of a Chow motive of the form $h(A^{r_j}_K)$, for suitable $r_j$'s. An analogous statement holds for more general coefficients $F$. 
\end{rem}

\section{The Hecke algebra} \label{hecke}

Keep the notations and the conventions of the previous section. In particular, we fix a Shimura variety $S_K$ of PEL-type, with underlying group $G$ and associated PEL datum $(V, B, *, \langle \cdot, \cdot \rangle, h_0)$. In this section, which is the heart of the paper, we are going to construct the motivic Hecke algebra $\mathcal{H}^M(G,K)$ lifting the classical, \emph{cohomological} Hecke algebra $\mathcal{H}(G,K)$ acting on $H^{\bullet}(S_K,V)$. In the first subsection we will review the latter algebra; in the second subsection, we will switch to the motivic setting.

The construction starts by considering, for any given $g \in G(\BA_f)$, the diagram of finite, étale morphisms
\begin{center}
\begin{tabular}{c}
\xymatrix{
& S_{K^g} \ar[dl]_{[\cdot 1]} \ar[dr]^{[\cdot g]} \\
S_K & & S_K
}
\end{tabular}
\end{center}
defined as follows. The quotient map
\[
[ \cdot 1]: S_{K^g}^{\an} \rightarrow S_K^{\an} 
\]
is induced by the inclusion $K^g \hookrightarrow K$. For every $g \in G(\BA_f)$, $g \neq 1$, there are holomorphic, finite maps (cfr. \cite[3.4]{Pin90})
\begin{align*} 
[ \cdot g]: S_{K^g}^{\an} & \rightarrow S_K^{\an} \\ \nonumber
G(\BQ)(x,h)K^g &\mapsto G(\BQ)(x,hg)K.
\end{align*}
Then (cfr. \cite[(3.4)]{Pin92}), these maps algebraize and give rise to maps
\begin{equation} \label{finmapg}
[\cdot 1], [\cdot g]: S_{K^g} \rightarrow S_K
\end{equation}
which are étale coverings if $K$ is small enough (which we will always suppose).

\subsection{The cohomological Hecke algebra} For the facts recalled in this section, we refer the reader to \cite[5.2 and 15.2-5]{HG24}.

Denote by $\mu^K$ either the Hodge or the $\ell$-adic canonical construction (Definitions \ref{Hodgecan}, \ref{lcan}). Fix $V \in \Rep(G_F)$. For any $g \in G(\BA_f)$, there are canonical isomorphisms (in the appropriate categories)
\begin{equation} \label{theta}
\xymatrix{
[\cdot 1]^* \mu^K(V) \ar[r]^{\overset{\theta^{\mbox{\tiny{-1}}}_1}{\simeq}} \ar@/_0.4cm/[rr]_{\theta} & \mu^{K^g}(V) \ar[r]^{\overset{\theta_g}{\simeq}} &  [\cdot g]^* \mu^K(V).
}
\end{equation}
In the analytic case, the isomorphism $\theta_g$ is given on the underlying total spaces (cfr. Eq. \eqref{analyticloc}) by
\begin{align}
\mu^{K^g}(V) & \rightarrow [\cdot g]^* \mu^K(V)=\mu^K(V) \times_{S_K,[\cdot g]} S_{K,g} \nonumber \\
[v,x,h] & \mapsto ([v,x,hg],[x,h]) \label{analyticmor}
\end{align}
where the symbols $[\cdot, \cdot, \cdot]$ or $[\cdot, \cdot]$ denote the appropriate equivalence classes.

We get isomorphisms
\begin{equation*}
\theta: H^{\bullet}(S_{K^g}, [\cdot 1]^* \mu^K(V)) \simeq H^{\bullet}(S_{K^g}, [\cdot g]^* \mu^K(V)).
\end{equation*}
Moreover, we have canonical adjunction morphisms
\begin{equation*}
[\cdot 1]^* : H^{\bullet}(S_K, \mu^K(V)) \rightarrow H^{\bullet}(S_K, [\cdot 1]_* [\cdot 1]^* \mu^K(V)) \simeq H^{\bullet}(S_{K^g}, [\cdot 1]^* \mu^K(V))
\end{equation*}
and, using that $[\cdot g]$ is finite and étale,
\begin{equation*}
[\cdot g]_* : H^{\bullet}(S_{K^g}, [\cdot g]^* \mu^K(V)) \simeq H^{\bullet}(S_K, [\cdot g]_* [\cdot g]^* \mu^K(V)) \rightarrow H^{\bullet}(S_K, \mu^K(V)).
\end{equation*}
The following compatibility between the isomorphisms $\theta_g$ and the polarization coming from the PEL datum is then immediate from the definitions. 
\begin{lm} \label{comp_pairing}
Let $V$ be the standard representation of $G$. For $g \in G(\BA_f)$, denote by $\langle \cdot,\cdot \rangle_g$ the pairing on 
\[
[\cdot g]^* \mu^{K}_H (V)=[\cdot g]^* H^1(A_{K})^{\vee}
\]
obtained by functoriality from the pairing of \eqref{pairinghodge} on $\mu^K_H(V)=H^1(A_K)^{\vee}$.
Then, for any $g$, 
\[
\langle \cdot,\cdot \rangle = \langle \theta_g(\cdot), \theta_g(\cdot) \rangle_g
\]
as pairings on $\mu^{K_g}_H (V)=H^1(A_{K_g})^{\vee}$.
\end{lm}
 
\begin{dfi} \label{heckeop}
The \emph{Hecke operator} $T_{K,g} \in \End H^{\bullet}(S_K, \mu^K(V))$ \emph{associated to} $g \in G(\BA_f)$ is defined by 
\begin{equation} \label{cohom_hecke}
T_{K,g} := [\cdot g]_* \circ \theta \circ [\cdot 1]^*. 
\end{equation}
\end{dfi} 
 
We denote by $\mathcal{C}_c^{\infty}(G,K)_F$ the \emph{Hecke algebra} of compactly supported, smooth $F$-valued functions on $G(\BA_f)$ which are bi-invariant under $K$, with product given by convolution. It is generated by characteristic functions $\one_{KgK}$ with $g \in G(\BA_f)$. Definition \ref{heckeop} provides us with a morphism of $F$-algebras (cfr. \cite[7.6]{GG12})
\begin{align} \label{hecke_arrow}
\begin{split}
\mathcal{C}_c^{\infty}(G,K)_F & \rightarrow \End H^{\bullet}(S_K, \mu^K(V)) \\
\one_{KgK} & \mapsto T_{K,g}.
\end{split}
\end{align}

\begin{dfi}
The \emph{cohomological Hecke algebra} (associated to $V$) is the $F$-subalgebra\footnote{The existence of the comparison isomorphism (after extending scalars appropriately) between Betti and étale $\ell$-adic cohomology, and the compatibility of the above-defined $\theta, [\cdot]^*$ and $[\cdot g]_*$ with the comparison isomorphism, imply that the definition, up to canonical isomorphism, is independent of the cohomology theory used.} 
\[
\mathcal{H}(G, K) \hookrightarrow
\End H^{\bullet}(S_K, \mu^K(V)) 
\]
defined as the image of the morphism \eqref{hecke_arrow}.
\end{dfi}

\begin{rem} \label{heckeintaction}
\item A Hecke algebra action on $H_c^{\bullet}(S_K,\mu^K(V))$ can be defined in an analogous way, so that the canonical map 
\[
H_c^{\bullet}(S_K,\mu^K(V)) \rightarrow H^{\bullet}(S_K,\mu^K(V))
\]
is equivariant with respect with the actions. Hence, we get a Hecke algebra action on $\mu^K(V)$-valued \emph{interior cohomology}, i.e. on
\begin{equation} \label{intcoh}
H_!^{\bullet}(S_K, \mu(V)):= \mbox{Im}(H_c^{\bullet}(S_K, \mu(V)) \rightarrow H^{\bullet}(S_K, \mu(V))).
\end{equation}
\end{rem}

\subsection{The motivic Hecke algebra}
Fix an element $g \in G(\mathbb{A}_f)$. There is a compact open subgroup $W$ of $V(\BA_f)$ such that the complex points of the universal abelian scheme $A_K$ over $S_K$ can be written as
\begin{equation*}
A_K(\BC)=V(\BQ) \rtimes G(\BQ) \backslash V(\BR) \times X \times V(\BA_f) \rtimes G(\BA_f) / W \rtimes K
\end{equation*}
where the semidirect product is defined by the standard representation of $G$ on $V$.
Analogously, seeing $g$ as an element of $V(\BA_f) \rtimes G(\BA_f)$, the complex points of the universal abelian scheme $A_{K^g}$ over $S_{K^g}$ are given by
\begin{equation*}
A_{K^g}(\BC)=V(\BQ) \rtimes G(\BQ) \backslash V(\BR) \times X \times V(\BA_f) \rtimes G(\BA_f) / W^g \rtimes K^g.
\end{equation*}
Now define the abelian schemes $A_{K,1}$ and $A_{K,g}$ over $S_{K^g}$ as the fiber products of $A_K$ and $S_{K^g}$ over $S_K$ along the morphisms, respectively, $[\cdot 1]$ and $[\cdot g]$ of equation \eqref{finmapg}.
These objects and morphisms fit in the following diagram, where all subdiagrams commute and the two lower subdiagrams are cartesian:

\begin{equation}
\label{Heckediag}
\begin{array}{c}
\xymatrix@C=2cm{
& A_{K^{g}} \ar[dl]_{f_1} \ar[d] \ar[dr]^{f_g} \\
A_{K,1} \ar[r] \ar[ddr] & S_{K^g} \ar@/_/[d]_{[\cdot 1]} \ar@/^/[d]^{[\cdot g]} & A_{K,g} \ar[l] \ar[ddl] \\
& S_{K} \\
& A_K \ar[u]
}
\end{array}
\end{equation}

The morphisms $f_1, f_g$ are isogenies; each $f_g$ is concretely described on $\BC$-points as
\begin{align}
A_{K^g}(\BC) & \rightarrow A_K(\BC) \times_{S_K(\BC),[ \cdot g]} S_{K^g}(\BC) \nonumber \\
[(v,x),(w,h)] & \mapsto ([(v,x),(wg,hg)],[(x,h)]) \label{analytichecke}
\end{align}
where square brackets denote the appropriate equivalence classes. 

\begin{def} \label{heckecorr}
Let $d$ denote the relative dimension of the abelian schemes under consideration. We define the morphism
\begin{equation} \label{phi_g}
\phi_g \in \Hom_{\DM_c(S_{K^g})}(h(A_{K,1}),h(A_{K,g}))=\CH^{d}(A_{K,1} \times_{S_{K^g}} A_{K,g})
\end{equation}
(cfr. \eqref{beilmorph}) as the (class of the) correspondence $\Gamma_{f_g} \circ \ ^t\Gamma_{f_1}$. 
\end{def}
The Chow-Künneth components are contravariantly functorial with respect to morphisms of abelian schemes, and functorial with respect to isogenies (Rem. \ref{propchowkun}.\ref{functchowkun},\ref{isochowkun}). Hence, we get from \eqref{phi_g} a morphism 
\begin{equation} \label{phi_g^1}
\phi^1_g : h^1( A_{K,1}) \rightarrow h^1(A_{K,g}).
\end{equation}

\begin{lm} \label{pullbh1}
There exist canonical isomorphisms
\begin{equation*}
[ \cdot 1]^* ( h^1(A_K)^{\otimes i} ) \simeq h^1(A_{K,1})^{\otimes i}
\end{equation*}
and 
\begin{equation*}
h^1(A_{K,g})^{\otimes i} \simeq [\cdot g]^* ( h^1(A_K)^{\otimes i} ).
\end{equation*}
\end{lm}
\begin{proof}
By proper base change, we have canonical isomorphisms 

\begin{equation} \label{bc1}
[ \cdot 1]^* h(A_K) \simeq h(A_{K,1})
\end{equation}
and
\begin{equation} \label{bc2}
h(A_{K,g}) \simeq [\cdot g]^* h(A_K).
\end{equation}
Since the characterization \eqref{canchowkun} of Chow-Künneth projectors shows immediately that the Chow-Künneth components are compatible with pullback, and since the functors $[ \cdot 1]^*$, $[\cdot g]^*$ are monoidal, these isomorphisms induce the isomorphisms in the statement. 
\end{proof}

\begin{rem} \label{can&bc}
By the above lemma and by Remark \ref{can&coh},
%shows that for $\mu^K$ equal to $\mu^K_H$ or $\mu^K_{\ell}$, the functors $[\cdot 1]^* \mu^K, \ [\cdot g]^* \mu^K$ send respectively (in the appropriate category)
%\[
%V^{\vee} \mapsto H^1(A_{K,1}), \ H^1(A_{K,g})
%\]
%Hence,
applying any realization functor $\rho$ yields canonical isomorphisms
\[
\rho(h^1(A_{K,1})) \simeq [ \cdot 1]^* \mu^K(V^{\vee}), \ \rho(h^1(A_{K,g})) \simeq [ \cdot g]^* \mu^K(V^{\vee})
\]
By the analytical description \eqref{analytichecke} of $f_1$, $f_g$  
%that under the base change isomorphisms 
%\[
%[\cdot 1]^* H^1 (A_K) \simeq H^1(A_{K,1}), \ \ [\cdot g]^* H^1 (A_K) \simeq H^1(A_{K,g})
%\]
the Hodge realization $\rho=\rho_H$ sends, under the above isomorphisms, the morphisms of motives $^t f_1$, $f_g$ to the isomorphisms of sheaves $\theta_1^{-1}$, resp. $\theta_g$ defined analytically in Eq. \eqref{analyticmor}.
\end{rem}

\begin{rem} \label{rem_algact}
Let $g \in G(\BA_f)$. By functoriality, the algebra $\CB_{i,F}$ of Definition \ref{Lef_alg} acts on both $[\cdot 1]^*h^1(A_{K})^{\otimes i}$ and $[\cdot g]^*h^1(A_{K})^{\otimes i}$. Hence, through the isomorphisms of Lemma \ref{pullbh1}, we can see $\CB_{i,F}$ as an algebra of endomorphisms of both $h^1(A_{K,1})^{\otimes i}$ and $h^1(A_{K,g})^{\otimes i}$.
\end{rem} 

We will need a compatibility statement between the morphisms $\phi_g^1$ and polarizations. Let $\square$ be either equal to $1$ or to $g$. By functoriality and by the isomorphisms of Lemma \ref{pullbh1}, the isomorphism \eqref{I} induces isomorphisms
\begin{equation} \label{iso_I}
I : h^1(A_{K,\square}) \simeq h^1(A_{K,\square})^{\vee}(-1)
\end{equation}
and hence, by adjunction, we get morphisms 
\[
p: h^1(A_{K,\square}) \otimes  h^1(A_{K, \square}) \rightarrow \BL
\]
and 
\[
\iota: \BL \rightarrow h^1(A_{K,\square}) \otimes h^1(A_{K, \square})
\]
analogously to \eqref{p}, \eqref{i}. 
\begin{lm} \label{commdiag_pol}
The diagram
\begin{center}
\begin{align}
\xymatrix{
h^1(A_{K,1}) \ar[d]_{\phi_g^1} \ar[r]^-{I}_-{\simeq} & h^1(A_{K,1})^{\vee} (-1) \\
h^1(A_{K,g}) \ar[r]^-{I}_-{\simeq} & h^1(A_{K,g})^{\vee} (-1) \ar[u]_{(\phi_g^1)^{\vee} (-1)}
}
\label{polcomm}
\end{align}
\end{center}
commutes.
\end{lm}
\begin{proof}
The commutativity can be checked after a base change to $\BC$, i.e. working in the category of smooth Chow motives over $S_{K^g, \BC}$, where, as a consequence of the isomorphism \eqref{isoEndH1} and of the equivalence of categories of Prop. \ref{abelhodge}, along with Remarks \ref{dualab_hodge} and \ref{dualab_mot}, we can equivalently show the commutativity of the corresponding diagram in the category of polarizable variations of $\BQ$-Hodge structure on $S_{K^g}(\BC)$, obtained by realization. By Remark \ref{can&bc}, our task amounts to showing the commutativity of the following diagrams of variations of Hodge structure on $S_{K^g}(\BC)$:
\[
\begin{array}{cc}
\xymatrix{
H^1(A_{K,1}) \ar[d]_{\theta^{-1}_1} \ar[r]^-{\rho(I)}_-{\simeq} & H^1(A_{K,1})^{\vee} (-1) \\
H^1(A_{K^g}) \ar[r]^-{\rho(I)}_-{\simeq} & H^1(A_{K^g})^{\vee} (-1) \ar[u]_{(\theta^{-1}_1)^{\vee} (-1)}
} 

& 

\xymatrix{
H^1(A_{K^g}) \ar[d]_{\theta_g} \ar[r]^-{\rho(I)}_-{\simeq} & H^1(A_{K^g})^{\vee} (-1) \\
H^1(A_{K,g}) \ar[r]^-{\rho(I)}_-{\simeq} & H^1(A_{K,g})^{\vee} (-1) \ar[u]_{(\theta_g)^{\vee} (-1)}
}

\end{array}
\]
Since $\rho(I)$ corresponds under adjunction to (the base change of) the pairing $\langle \cdot, \cdot \rangle$ of \eqref{pairinghodge}, the desired commutativity follows from Lemma \ref{comp_pairing}. 
\end{proof}

We have now all the ingredients to state and prove the crucial result underlying the construction of the motivic Hecke algebra. 
\begin{thm} \label{crucialcomm}
Let $A_K \rightarrow S_K$ be the universal abelian scheme over a PEL Shimura variety with underlying group $G$. Let $g \in G(\BA_f)$ and consider the abelian schemes $A_{K,1}$ and $A_{K,g}$ of diagram \eqref{Heckediag} and the morphism $\phi^1_g$ of equation \eqref{phi_g^1}. Then for any integer $i$, for any element $\beta$ of the algebra $\CB_{i,F}$ of Definition \ref{Lef_alg} acting as in Remark \ref{rem_algact}, the equality 
\[
(\phi_g^{1})^{\otimes i} \circ \beta = \beta \circ (\phi_g^{1})^{\otimes i}
\]
holds in 
\[
\Hom_{\DM_c(S_{K^g})}((h^1(A_{K,1}))^{\otimes i},(h^1(A_{K,g}))^{\otimes i}).
\]
\end{thm}

\begin{proof}

It will be enough to show that the desired equality holds for any $\beta$ in the explicit set of generators of $\CB_{i,F}$ provided by Def. \ref{Lef_alg}.

Commutativity of $(\phi_g^{1})^{\otimes i}$ with elements $\beta$ of the symmetric group $\mathfrak{S}_i$ is clear. Hence, let us fix $b \in B$ and turn our attention to commutativity with 
\[
\beta=b \otimes \Id^{\otimes i-1}_{h^1(A_{K,\square})}
\]
with $\square$ being either equal to $1$ or to $g$. Let us simultaneously see the algebra $B$ as an algebra of endomorphisms of each of the three objects $h^1(A_{K,1})$, $h^1(A_{K,g})$ and $h^1(A_{K_g})$. With this convention, given the definition of $\phi_g$ in \eqref{phi_g}, the desired commutativity relation will follow once we prove that $b$ commutes with $^t\Gamma_{f_1}$ and $\Gamma_{f_g}$. The morphisms $f_1$ and $f_g$ being isogenies, we have that $b$ commutes with $^t\Gamma_{f_1}$, $ ^t\Gamma_{f_g}$ if and only if it commutes with $\Gamma_{f_1}$, $\Gamma_{f_g}$. But the validity of the latter assertion is immediate from the analytical description \eqref{analytichecke} of the morphisms $f_1$, $f_g$. 

Finally, let again $\square$ be either equal to $1$ or to $g$ and let us show commutativity of $(\phi_g^{1})^{\otimes i}$ with $\beta =P \otimes \Id^{\otimes i-2}_{\Fh^1(A_{K,\square})}$ and thus conclude the proof. We have to show that the projector $P$ (as defined in Def. \ref{Lef_alg}) commutes with $(\phi_g^1)^{\otimes 2}$, i.e., with the notations of Lemma \ref{commdiag_pol}, we have to show the identity 
\begin{equation} \label{commprojpol}
\iota \circ p \circ (\phi_g^1 \otimes \phi_g^1) = (\phi_g^1 \otimes \phi_g^1) \circ \iota \circ p.
\end{equation}
This is the content of Lemma \ref{lemmacalc} below. 
\end{proof}

\begin{lm} \label{lemmacalc}
The identity \eqref{commprojpol} holds. 
\end{lm}
\begin{proof}
In the following, we will denote 
\[
I_1 : h^1(A_{K,1}) \simeq h^1(A_{K,1})^{\vee}(-1)
\]
and 
\[
I_g : h^1(A_{K,g}) \simeq h^1(A_{K,g})^{\vee}(-1)
\]
the isomorphisms defined in \eqref{iso_I}, and we will use the the isomorphism $adj$ from \eqref{adj}. We then have
\begin{align*}
& \iota \circ p \circ (\phi_g^1 \otimes \phi_g^1) \quad \\
& \mbox{(by definition of} \ \iota \ \mbox{and} \ p) \\
= \quad & adj(I_g^{-1}) \circ adj^{-1}(I_g) \circ \phi_g^1 \otimes \phi_g^1 \\
& \mbox{(by commutativity of \eqref{polcomm})} \\
= \quad & adj (\phi_g^1 \circ I_1^{-1} \circ (\phi_g^1)^{\vee} (-1)) \circ adj^{-1}(I_g) \circ \phi_g^1 \otimes \phi_g^1 \quad \\
& \mbox{(by Lemma \ref{adjformulae} in its instance \ref{f1}, choosing} \ \mu=\phi_g^1 \ \mbox{and} \ \lambda=I_1^{-1} \circ (\phi_g^1)^{\vee} (-1) ) \\
= \quad & (\Id_{h^1(A_{K,g})} \otimes \phi_g^1) \circ adj(I_1^{-1} \circ (\phi_g^1)^{\vee} (-1)) \circ adj^{-1}(I_g) \circ \phi_g^1 \otimes \phi_g^1 \quad \\
& \mbox{(by Lemma \ref{adjformulae} in its instance \ref{f2}, choosing} \ \mu=I_1^{-1} \ \mbox{and} \ \lambda = (\phi_g^1)^{\vee}) \\
= \quad & (\Id_{h^1(A_{K,g})} \otimes \phi_g^1) \circ (\phi_g^1 \otimes \Id_{h^1(A_{K,1})}) \circ adj(I_1^{-1}) \circ adj^{-1}(I_g) \circ \phi_g^1 \otimes \phi_g^1 \quad \\
& (\mbox{by factorizing} \ \phi^1_g \otimes \phi^1_g = (\Id_{h^1(A_{K,g})} \otimes \phi_g^1) \circ (\phi_g^1 \otimes \Id_{h^1(A_{K,1})}) \ \mbox{and by definition of} \ \iota) \\
= \quad & \phi^1_g \otimes \phi^1_g \circ \iota \circ adj^{-1}(I_g) \circ (\Id_{h^1(A_{K,g})} \otimes \phi_g^1) \circ (\phi_g^1 \otimes \Id_{h^1(A_{K,1})}) \quad \\
& \mbox{(by Lemma \ref{adjformulae} in its instance \ref{f3}, choosing} \ \mu=I_g \ \mbox{and} \ \lambda= \phi_g^1 ) \\
= \quad & \phi^1_g \otimes \phi^1_g \circ \iota \circ adj^{-1}(I_g \circ \phi_g^1) \circ (\phi_g^1 \otimes \Id_{h^1(A_{K,1})}) \quad \\
& \mbox{(by Lemma \ref{adjformulae} in its instance \ref{f4}, choosing} \ \mu=(\phi_g^1)^{\vee} \ \mbox{and} \ \lambda = I_g \circ \phi^1_g) \\
= \quad & \phi^1_g \otimes \phi^1_g \circ \iota \circ adj^{-1}( (\phi_g^1)^{\vee} (-1) \circ I_g \circ \phi_g^1) \quad \\
& \mbox{(by commutativity of \eqref{polcomm})} \\
= \quad & \phi^1_g \otimes \phi^1_g \circ \iota \circ adj^{-1}(I_1) \quad \\
& \mbox{(by definition of} \ p) \\
\end{align*}
\begin{align*}
= \quad & (\phi_g^1 \otimes \phi_g^1) \circ \iota \circ p.
\end{align*}
\end{proof}

\begin{rem} 
For any positive integer $i$, consider the morphism
\begin{equation} \label{relhecke}
\phi^{1,i}_g : [ \cdot 1]^* ( h^1(A_K)^{\otimes i} )  \rightarrow [ \cdot g]^* (h^1(A_{K})^{\otimes i}).
\end{equation}
as the one obtained from the morphism induced by \eqref{phi_g^1}
\begin{equation} \label{tensor_phi}
(\phi^1_g)^{\otimes i} : h^1(A_{K,1})^{\otimes i} \rightarrow h^1(A_{K,g})^{\otimes i}
\end{equation}
by composing with the isomorphisms of Lemma \ref{pullbh1}. For any idempotent element $e$ of the algebra $\CB_{i,F}$ of Definition \ref{Lef_alg}, call $\CN^e$ the corresponding direct factor in $\CHM^s(S_K)_F$
\[
\CN^e \hookrightarrow h^1(A_K)^{\otimes i}.
\]
Then, Theorem \ref{crucialcomm} implies that for any $g \in G(\BA_f)$, the morphism $\phi_g^{1,i}$ restricts to a morphism 
\begin{equation} \label{phi_g^e}
\phi^e_g : [ \cdot 1]^* \CN^e \rightarrow [ \cdot g]^* \CN^e.
\end{equation}
\end{rem}

\begin{cor} \label{phi_V}
Let $V$ be any object in $\Rep_F(G)$ and consider the object $\tilde{\mu}(V)$ in $\CHM^s(S_K)_F$, where $\tilde{\mu}$ is the motivic canonical construction functor of Theorem \ref{constmot}. Then, for any $g \in G(\BA_f)$, there exists a canonical morphism in $\CHM^s(S_{K^g})_F$
\[
\phi_g : [ \cdot 1]^* \tilde{\mu}(V) \rightarrow [ \cdot g]^* \tilde{\mu}(V)
\]
whose realizations coincide with the morphism $\theta$ of Eq. \eqref{theta}.
\end{cor}
\begin{proof}
By Remark \ref{factpowab}, the object $\tilde{\mu}(V)$ is a direct sum of Tate twists of direct factors of motives $h(A_K^{r})$. By the Chow-Künneth decomposition \eqref{CKdec} of the abelian scheme $A_K^r$, we may suppose that each such direct factor is cut out by an idempotent in the subalgebra $\CB_{i,r,F}$ of $\CB_{i,F}^{\oplus r}$ of Definition \ref{Lef_alg^r} (for a suitable $i$). For each such idempotent $e$, we have the morphisms $\phi^e_g$ of Equation \eqref{phi_g^e}, whose direct sum gives the desired morphism $\phi_g$ as in the statement (notice that the commutativity of diagram \eqref{polcomm} assures the compatibility of $\phi_g$ with the identifications $h^1(A_K) \simeq h^1(A_K)^{\vee}(-1)$ being made to get the conclusion of Remark \ref{factpowab}). The statement about realizations comes from Remark \ref{can&bc}. 
\end{proof}

Let $g \in G(\BA_f)$ and $V \in \Rep_F(G)$. We have adjunction morphisms in $\CHM^s(S_K)_F$
\[
a_1: \tilde{\mu}(V) \rightarrow [\cdot 1]_{*} [ \cdot 1]^* \tilde{\mu}(V), \ \ a_g: [ \cdot g]^* [\cdot g]_{*} \tilde{\mu}(V) = [ \cdot g]^! [\cdot g]_{!} \tilde{\mu}(V) \rightarrow \tilde{\mu}(V).
\]
Let $s:S_K \rightarrow \Spec E$ denote the structure morphism and apply $s_*$ to $a_1$ and $a_g$. Then one gets morphisms in $\DM_c(E)_F$
\begin{equation} \label{heckepart1}
s_*(a_1):s_* \tilde{\mu}(V) \rightarrow (s \circ [\cdot 1])_* [ \cdot 1]^* \tilde{\mu}(V), \ s_*(a_g): (s \circ [\cdot g])_* [\cdot g]^* \tilde{\mu}(V) \rightarrow s_* \tilde{\mu}(V).
\end{equation}
By applying $(s \circ [\cdot 1])_*=(s \circ [\cdot g])_*$ to the morphism $\phi_g$ of Corollary \ref{phi_V}, one gets another morphism in $\DM_c(E)_F$
\begin{equation} \label{heckepart2}
(s \circ [\cdot 1])_*(\phi_g): (s \circ [\cdot 1])_* [\cdot 1]^* \tilde{\mu}(V) \rightarrow (s \circ [\cdot g])_* [\cdot g]^* \tilde{\mu}(V).
\end{equation}

\begin{dfi} \label{mothecke} Let $V \in \Rep_F(G)$. The \emph{Hecke correspondence} on $s_* \tilde{\mu}(V)$ associated to $g \in G(\BA_f)$ is the element $KgK \in \End_{\DM_c(E)_F}(s_* \tilde{\mu}(V))$ defined by
\begin{equation}
KgK:= s_*(a_g) \circ (s \circ [\cdot 1])_*(\phi_g) \circ s_*(a_1) : s_* \tilde{\mu}(V) \rightarrow s_* \tilde{\mu}(V). 
\end{equation}
\end{dfi}

\begin{dfi} \label{motheckealg}
The \emph{motivic Hecke algebra} $\mathcal{H}^{M}(G, K)$ is the subalgebra
\[
\mathcal{H}^{M}(G, K) \hookrightarrow \End_{\DM_c(E)_F}(s_* \tilde{\mu}(V))
\]
generated by the correspondences $KgK$ of Definition \ref{mothecke}.
\end{dfi}

\begin{rem}
The motivic Hecke algebra $\mathcal{H}^{M}(G, K)$ defined above depends on the representation $V \in \Rep_F(G)$. Strictly speaking, the datum $V$ should be included in the notation, but we omit it for ease of reading.   
\end{rem}

\begin{rem}
We thank the referee for suggesting the following observation. If $S_K$ is \emph{any} Shimura variety, i.e. not necessarily of PEL type, associated with a reductive group $G$, then for any $g \in G(\BA_f)$ one can always define the morphisms $[\cdot 1]$ and $[\cdot g]$. Taking the constant coefficient system $\one_{S_K}$, and letting $s$, resp. $s_g$ be the structure morphisms of $S_K$, resp. $S_{K^g}$, one has morphisms in $\DM_c(E)_F$
\[
s_*(a_1):s_* \one_{S_K} \rightarrow (s \circ [\cdot 1])_* \one_{S_{K^g}} \simeq s_{g,*} \one_{S_{K^g}}
\] 
and 
\[
s_*(a_g): s_{g,*} \one_{S_{K^g}} \simeq (s \circ [\cdot g])_* \one_{S_{K^g}} \rightarrow s_* \one_{S_K}.
\]
These can be composed to get a Hecke correspondence
\[
KgK:= s_*(a_g) \circ s_*(a_1) : s_* \one_{S_K} \rightarrow  s_* \one_{S_K}
\]
hence allowing to define a motivic Hecke algebra $\mathcal{H}^{M}(G, K)$ for \emph{any} Shimura variety with constant coefficients. 
\end{rem}

We are then ready to formulate, and to easily deduce, our main result.
\begin{thm} \label{realhecke}
The motivic Hecke algebra $\mathcal{H}^{M}(G, K)$ of Definition \ref{motheckealg} enjoys the following property. Any realization functor $\rho$ induces an $F$-algebra epimorphism 
\[
\rho : \mathcal{H}^{M}(G, K) \twoheadrightarrow \mathcal{H}(G, K)
\]
\end{thm}
\begin{proof}
It is an immediate consequence of the definitions, of the compatibility of realizations with the six functors, and of Corollary \ref{phi_V}. 
\end{proof}

\begin{rem} \label{morehecke}
\begin{enumerate}[wide, labelwidth=!, labelindent=0pt, label=(\arabic*)]
\item By using the fact that $[\cdot 1]_*=[\cdot 1]_!$ and $[\cdot g]_*=[\cdot g]_!$, one can define \emph{Hecke correspondences with compact support}
\begin{equation}
KgK_c:= s_!(a_g) \circ (s \circ [\cdot 1])_!(\phi_g) \circ s_!(a_1) : s_! \tilde{\mu}(V) \rightarrow s_! \tilde{\mu}(V)
\end{equation}
and a \emph{motivic Hecke algebra with compact support} $\mathcal{H}_c^{M}(G, K)$ acting on $s_! \tilde{\mu}(V)$. Then, the canonical map $s_! \tilde{\mu}(V) \rightarrow s_* \tilde{\mu}(V)$ is equivariant with respect to the Hecke actions.
\item \label{differences} (\textbf{\emph{Comparison with previous constructions.}}) Let $s:S_K \rightarrow \Spec E$ be \emph{any} Shimura variety (not necessarily of PEL type) and let $M_K$ be a \emph{mixed Shimura variety} (\cite[Def. 3.1]{Pin90}) equipped with a morphism $M_K \rightarrow S_K$ giving to $M_K$ the structure of an abelian scheme over $S_K$. Given any direct factor $\CN$ of $h_{S_K}(M_K)$ in $\CHM(S_K)_F$, an endomorphism $KgK \in \End_{\DM_c(E)_F}(s_*\CN)$ is defined in \cite[pp. 591-592]{Wil17} for any $g \in G(\BA_f)$, by a formula formally identical to the one of Definition \ref{mothecke}, up to replacing our morphism $\phi_g$ by a morphism $\tilde{\phi}_g$ (the notation is ours) defined in a different way. 

Let us explain the difference by looking at the basic case of interest here, i.e. $M_K$ being equal to a power $A^r_K$ of the universal abelian scheme $A_K$ over a PEL Shimura variety $S_K$. There is a morphism $\phi_g : [\cdot 1]^* h(A^r_K) \rightarrow [\cdot g]^* h(A^r_K)$ defined analogously to \eqref{phi_g}. Then $\tilde{\phi}_g$ is obtained by pre-composing $\phi_g$ with the inclusion $[\cdot 1]^* \CN \hookrightarrow [\cdot 1]^*h(A^r_K)$ and post-composing it with the projection $[\cdot g]^*h(A^r_K) \twoheadrightarrow [\cdot g]^* \CN$. One could define the motivic Hecke algebra as the one generated by correspondences $KgK$ defined in this way; suppose one does this, obtaining algebras $\mathcal{H}^M(G,K,\mathcal{N})$ for each $\mathcal{N}$. Then, the following problem would arise. Choose any realization $\rho$ and for the sake of this remark, keep, for the cohomological Hecke algebra as well, the notation $\mathcal{H}(G,K,\rho(\mathcal{N}))$ remembering the coefficient system $\rho(\mathcal{N})$. There is an idempotent $e$ cutting out $\CN$ in $h(A_K^r)$, inducing an idempotent $\mathfrak{E}:=s_*(e)$ acting on the absolute motive $s_* h(A_K^r)$ and on its cohomology. Then, $\rho$ induces an algebra epimorphism 
\[
\rho: \mathcal{H}^M(G,K,\CN) \twoheadrightarrow \mathfrak{E} \circ \mathcal{H}(G,K,H^{\bullet}(A_K^r)) \circ \mathfrak{E}
\]
but there is a priori no reason for the arrow 
\[
\mathfrak{E} \circ \mathcal{H}(G,K,H^{\bullet}(A_K^r)) \circ \mathfrak{E} \rightarrow \mathcal{H}(G,K,\rho(\mathcal{N}))
\] 
to even be an \emph{algebra homomorphism}.
\end{enumerate}
\end{rem}

\section{The Hecke action on interior motives of PEL Shimura varieties}\label{hecke_interior}

We keep the setting of the previous section. In particular, $s:S_K \rightarrow \Spec E$ is a PEL Shimura variety of underlying group $G$ and for any $V \in \Rep_F(G)$, we have the smooth Chow motive $\tilde{\mu}(V)$ over $S_K$ of Theorem \ref{constmot} at our disposal. We will denote by $\mu$ either the Hodge (Def. \ref{Hodgecan}) or the $\ell$-adic (Def. \ref{lcan}) canonical construction functors on $S_K$. 

In \cite{Wil19}, Wildeshaus has shown that under certain conditions, it is possible to construct a \emph{functorial} factorization of the canonical morphism $s_! \tilde{\mu}(V) \rightarrow s_* \tilde{\mu}(V)$ in $\DM_c(E)_F$
\begin{center}
\begin{tabular}{c}
\xymatrix{
& \Gr_0(s_* \tilde{\mu}(V)) \ar[dr] & \\
s_! \tilde{\mu}(V) \ar[ur] \ar[rr] & & s_* \tilde{\mu}(V)
}
\end{tabular}
\end{center}
in such a way that $\Gr_0(s_* \tilde{\mu}(V))$ is a \emph{Chow} motive, called \emph{interior motive} of $S_K$ with values in $\tilde{\mu}(V)$. Its name comes from the fact that it realizes to $\mu(V)$-valued \emph{interior cohomology} of $S_K$ \eqref{intcoh}.

The aim of this final section is twofold. First, we explain that when it exists, the interior motive carries an action of the motivic Hecke algebra constructed in the previous section. Second, we briefly review a list of cases in which the existence of the interior motive is indeed known, so that our main result applies.  

Let us start by recalling the principle of construction, by adopting the point of view of \emph{intersection motives}. This requires a choice of compactification of $S_K$, which we take to be the \emph{Baily-Borel} compactification $S_K^*$ (see e.g. \cite[(3.5)]{Pin92}), a projective variety over $E$ equipped with morphisms
\[
j : S_K \hookrightarrow S_K^* \hookleftarrow \partial S_K^* : i
\]
with $j$ the open, dense immersion of $S_K$, and $i$ the closed immersion of the \emph{boundary} $\partial S_K^*$. 
\begin{dfi} \label{avoidweight}
For any object $M \in \DM_c(S_K)_F$, we define $\partial M$ as the constructible motive
\[
\partial M:= i^* j_* M \in \DM_c(\partial S_K^*)_F.
\]
The full subcategory 
\[
\DM(S_K)_{F,\partial w \neq 0,1}
\]
of $\DM(S_K)_F$ is defined as having objects those $M \in \DM(S_K)_F$ such that $\partial M$ \emph{avoids weights 0 and 1}, in the sense of the motivic weight structure on $\DM_c(\partial S_K^*)_F$ (see \cite[Sec. 2]{Wil19} for the relevant definitions). We write $\CHM(S_K)_{F,\partial w \neq 0,1}$ for its intersection with $\CHM(S_K)_{F}$.
\end{dfi} 

\begin{rem} \label{inters_mot}
\begin{enumerate}[wide, labelwidth=!, labelindent=0pt, label=(\arabic*)]
\item \label{ab_mot} By combining \cite[Def. 2.4, Rem. 2.6(d)]{Wil19} with \cite[Thm. 7.2]{Wil17}, one knows that there exists a functor 
\[
j_{!*} : \CHM(S_K)_{F,\partial w \neq 0,1} \rightarrow \CHM(S_K^*)_F
\]
enjoying in particular the following property. Consider the full subcategory 
\[
\CHM(S_K)^{\ab}_{F} \hookrightarrow \CHM(S_K)_F
\]
consisting of those relative Chow motives which are the pullback to $S_K$ of a motive \emph{of abelian type} over $S_K^*$ (\cite[Def. 4.1]{Wil19}). Then, the restriction of $j_{!*}$ to the intersection of $\CHM(S_K)_{F,\partial w \neq 0,1}$ with $\CHM(S_K)^{\ab}_{F}$ realizes to the \emph{intermediate extension} functor defined in the theory of \emph{perverse sheaves}. We refer to \cite[Thm. 7.2]{Wil17} for the precise formulation of the latter realization statement. Note that by \cite[Thm. 8.5]{Wil17}, the objects $\tilde{\mu}(V)$ belong to $\CHM(S_K)^{\ab}_{F}$ for any $V$ in $\Rep_F(G)$.
\item \label{pure_mot}
By the properties of the motivic weight structure, pushforward along proper morphisms sends weight zero objects to weight zero objects. Hence, if $\bar{s} : S_K^* \rightarrow \Spec E$ denotes the structure morphism, the functor $j_{!*}$ of the previous point allows us to define a functor 
\[
\bar{s}_* j_{!*} : \CHM(S_K)_{F,\partial w \neq 0,1} \rightarrow \CHM(E)_F.
\]
By \cite[Thms. 3.4 and 3.5]{Wil19}, it sits in a commutative diagram of natural transformations 
\begin{center}
\begin{tabular}{c}
\xymatrix{
& \bar{s}_* j_{!*} \ar[dr] & \\
s_! \ar[ur] \ar[rr] & & s_* 
}
\end{tabular}
\end{center}
between functors $\DM(S_K)_{F,\partial w \neq 0,1} \rightarrow \DM(E)_F$. It has moreover the property that, for any $M \in \CHM(S_K)_{F,\partial w \neq 0,1}$, any endomorphism of $s_* M$ or of $s_! M$ induces an endomorphism of $\bar{s}_*  j_{!*} M$, compatibly with composition. 
\end{enumerate}
\end{rem}

These facts, along with \cite[Rmk. 3.13(b)]{Wil19}, yield the following. 

\begin{thm} \label{intmotive_thm}
Let $s:S_K \rightarrow \Spec E$ be a Shimura variety of PEL type, of underlying group $G$. Let $V \in \Rep_F(G)$ and suppose that $\tilde{\mu}_F(V)$ belongs to the subcategory $\CHM(S_K)_{F,\partial w \neq 0,1}$ of Def. \ref{avoidweight}. Let $\bar{s} : S_K^* \rightarrow \Spec E$ be the structure morphism from the Baily-Borel compactification of $S_K$. Then 
\begin{enumerate}[wide, labelwidth=!, labelindent=0pt, label=(\arabic*)]
\item the Chow motive $\bar{s}_* j_{!*} \tilde{\mu}(V)$ (see Remark \ref{inters_mot}.\ref{pure_mot}) sits in a functorial factorization in $\DM_c(E)_F$
\begin{center}
\begin{tabular}{c}
\xymatrix{
& \bar{s}_* j_{!*} \tilde{\mu}(V) \ar[dr] & \\
s_! \tilde{\mu}(V) \ar[ur] \ar[rr] & & s_* \tilde{\mu}(V)
}
\end{tabular}
\end{center}
such that any endomorphism of $s_* \tilde{\mu}(V)$ or of $s_! \tilde{\mu}(V)$ induces an endomorphism of $\bar{s}_* j_{!*} \tilde{\mu}(V)$, compatibly with composition;
\item there is a canonical isomorphism between $\mu(V)$-valued interior cohomology of $S_K$ \eqref{intcoh} and $\mu(V)$-valued intersection cohomology of $S_K^*$. For any realization $\rho$, there is a canonical isomorphism  
\[
\rho(\bar{s}_* j_! \tilde{\mu}(V)) \simeq H^{\bullet}_!(S_K, \mu(V)).
\]
\end{enumerate}
\end{thm}

Point (2) of the above theorem motivates the following definition. 
\begin{dfi} \label{intmotive_def}
Let $s:S_K \rightarrow \Spec E$ be a Shimura variety of PEL type, of underlying group $G$. Let $V \in \Rep_F(G)$ and suppose that $\tilde{\mu}_F(V)$ belongs to $\CHM(S_K)_{\partial w \neq 0,1}$. The Chow motive $\bar{s}_* j_{!*} \tilde{\mu}(V)$ over $E$ of Theorem \ref{intmotive_thm} is called \emph{interior motive} of $S_K$, with values in $\mu(V)$, and denoted by\footnote{The notation $\Gr_0$ refers to the fact that under the weight avoidance assumption at the boundary, the object $\bar{s}_* j_{!*} \tilde{\mu}(V)$ has to be thought, in a sense made precise by the theory of weight structures, as the \q{lowest weight-graded quotient} of $s_* \tilde{\mu}(V)$. One manifestation of this phenomenon is the following: under our assumptions, the Hodge realization of $\bar{s}_* j_{!*} \tilde{\mu}(V)$ can be identified as well with the lowest weight-graded quotient (in fact, subspace) of the Hodge-theoretic weight filtration on $H^{\bullet}(S_K,\mu(V))$.}
\[
\Gr_0 s_* \tilde{\mu}(V).
\]
\end{dfi}

Our final theorem reads then as follows.
\begin{thm} \label{motheckeaction}
Let $s:S_K \rightarrow \Spec E$ be a PEL Shimura variety, of underlying group $G$. Let $V \in \Rep_F(G)$ and suppose that $\tilde{\mu}_F(V)$ belongs to the subcategory $\CHM(S_K)_{F,\partial w \neq 0,1}$ of Definition \ref{avoidweight}.

Then, there is a canonical action of the motivic Hecke algebra $\mathcal{H}^{M}(G, K)$ of Def. \ref{mothecke} on the interior motive $\Gr_0 s_* \tilde{\mu}(V)$ of Def. \ref{intmotive_def}, realizing to the Hecke algebra action of Remark \ref{heckeintaction} on $\mu(V)$-valued interior cohomology of $S_K$.
\end{thm}

\begin{proof}
It suffices to combine Theorem \ref{realhecke} with Thm. \ref{intmotive_thm} above. 
\end{proof}

\begin{rem} \label{intersection_motive}
\begin{enumerate}[wide, labelwidth=!, labelindent=0pt, label=(\arabic*)] \item Consider the notations of Theorem \ref{intmotive_thm}. By \cite[Thms. 0.1-0.2]{Wil17}, a Chow motive $j_{!*} \tilde{\mu}(V)$ over $S_K^*$, extending $\tilde{\mu}(V)$ and realizing to the intermediate extension of $\mu(V)$ in the perverse sheaf-theoretic sense, is known to exist for any $V$, even without the weight avoidance hypothesis at the boundary, and to be such that any endomorphism of $s_* \tilde{\mu}(V)$ induces an endomorphism of $\bar{s}_* j_{!*} \tilde{\mu}(V)$. This extension property can be then applied, in particular, to the elements of the motivic Hecke algebra\footnote{In fact, the results of \emph{loc. cit.} apply in far greater generality: one gets relative Chow motives $j_{!*} \CN$, enjoying the above properties, for any direct factor $\CN$ of any motive $h_{S_K}(M_K)$ of the form discussed in Remark \ref{morehecke}.\ref{differences}. One then gets Hecke-type endomorphisms $KgK$ of $\bar{s}_* j_{!*} \CN$ as defined in \emph{op. cit.}, pp. 591-592.} $\mathcal{H}^{M}(G, K)$. However, we do not know whether this provides an \emph{algebra} action of $\mathcal{H}^{M}(G, K)$ on $\bar{s}_* j_{!*} \tilde{\mu}(V)$. The problem comes from the lack of functoriality of the association 
\[
\tilde{\mu}(V) \mapsto j_{!*} \tilde{\mu}(V)
\]
when one does not assume the weight avoidance hypothesis of Theorem \ref{motheckeaction}. 
\item \cite[Cor. 8.8 (b)]{Wil17} gives conditions under which the problem raised by the previous point can be solved. According to \cite[Rmk. 8.9 (a) and (c)]{Wil17}, one may expect these conditions to be met if $V$ is an \emph{irreducible} representation of $G_F$.
\end{enumerate}
\end{rem}

We want to conclude with a list of cases where the weight assumption of Theorem \ref{motheckeaction} is known to hold. The first tautological situation (but still giving rise to interesting Hecke-equivariant Chow motives) is when $S_K$ is \emph{projective}, for then the weight assumption is empty. In order to give the statements for the non-projective case, choose the field of coefficients $F$ in such a way that the reductive group $G_F$ is \emph{split}, so that we have the theory of highest weights at our disposal. Denote by $V_{\lambda}$ an irreducible representation of $G_F$ of highest weight $\lambda$. Let us then write 
\[
^{\lambda} \CV
\]
for the object $\tilde{\mu}(V_{\lambda})$ of $\CHM(S_K)_F$. In \cite{Wil19}, the following is asked. 

\begin{ques}{\cite[Question 5.13]{Wil19}} \label{reg_question}
Let $S_K$ be non-projective. Suppose the highest weight $\lambda$ to be \emph{regular}. Does the Chow motive $^{\lambda} \CV$ over $S_K$ belong to $\CHM(S_K)_{F,\partial w \neq 0,1}$ ?
\end{ques}

Thanks to the work of Ancona, of the present author, of Cloître, and of Wildeshaus, the following is known (see \cite[Section 5]{Wil19} for details and for the precise references).
\begin{thm} \label{knowncases}
The answer to Question \ref{reg_question} is positive when $S_K$ is associated to the standard PEL Shimura datum having as $G$ one of the following groups\footnote{In order to satisfy the \emph{rational similitude factor} condition in the definition of a rational PEL datum, one actually considers suitable subgroups of the given $G$'s, with same derived subgroup and smaller center.} :
\begin{enumerate}[wide, labelwidth=!, labelindent=0pt, label=(\arabic*)]
\item $\Res_{F \vert \BQ} \GL_{2,F}$, for $F$ a totally real number field (\emph{Hilbert modular varieties}, hence \emph{modular curves} when $F=\BQ$);
\item $\Res_{F^{+} \vert \BQ} \GU(V,J)$, for $F$ a CM field with maximal totally real subfield $F^{+}$, $V$ a 3-dimensional $F$-vector space $V$, and $J$ a $F$-valued hermitian form $J$ on $V$ of signature $(2,1)$ at each archimedean place (\emph{Picard modular varieties}); 
\item $\Res_{F \vert \BQ} \GSp_{4,F}$, for $F$ a totally real number field (\emph{Hilbert-Siegel varieties of genus two}, hence \emph{Siegel threefolds} when $F=\BQ$).
\end{enumerate}
\end{thm}

\begin{rem} \label{weights}
In cases (1) and (3) when $F \neq \BQ$, and in case (2) when $F^+ \neq \BQ$, weaker conditions on $\lambda$ than regularity are known to be sufficient for the weight assumption to be satisfied. For example, the notion of \emph{corank}, controlling the existence of sections of automorphic bundles on $S_K$, which restrict non-trivially to strata of $\partial S^*_K$ of a given dimension, arises when characterizing the validity of the weight assumption in case (3) (\cite{Cav19}). Once again, we refer to \cite[Section 5]{Wil19} for a global discussion.

\end{rem}

\begin{cor} \label{motheckeaction_knowncases}
The hypotheses, and hence the conclusion, of Theorem \ref{motheckeaction} hold whenever $S_K$ is one of the Shimura varieties in Theorem \ref{knowncases} and $V=V_{\lambda}$ with $\lambda$ a weight of $G$ as in Remark \ref{weights} (for example, a \emph{regular} weight). 
\end{cor}

\begin{rem} 
Theorem \ref{motheckeaction} was already known when the PEL Shimura variety is a Hilbert modular variety (\cite[Cor. 3.8]{Wil12}), by means of the explicit formulae, provided in \emph{op. cit.}, for the idempotents cutting out the relevant direct factors $^\lambda \CV$ from the motive of a suitable power of $A_K$. 
\end{rem}

\section*{Acknowledgements}
I started getting interested in the problems studied in this paper some time ago, during my PhD thesis (which contains a rough form of the results presented here) at the Université Paris 13, under the direction of J. Wildeshaus. I thank him heartily for introducing me to these subjects and for our exchanges at that time and during the years. I also thank F. Brunault, D. Loeffler, S. Morel and S. Zerbes for interesting questions and exchanges around the content of this article.  Special thanks to G. Ancona for crucial discussions and suggestions around his results on the decomposition of abelian schemes, which are of course the key input for everything done here, and for invaluable, constant encouragement. The almost-final version of this text was completed while I was a member of the Laboratoire de Mathématiques d'Orsay : I acknowledge the excellent working conditions there, as well as the support of the Fondation Mathématique Jacques Hadamard. Last but not least, thanks to the referee for their careful reading and thoughtful comments.  

\nocite{}
\bibliographystyle{abbrv}
\bibliography{Library}

\end{document}